%% file: Brito_and_Hoffman_arXiv.tex
\documentclass[graybox]{svmult}

\usepackage{mathptmx}       
\usepackage{helvet}         
\usepackage{courier}        
\usepackage{type1cm}        
\usepackage{makeidx}         
\usepackage{graphicx}        
\usepackage{multicol}        
\usepackage[bottom]{footmisc}

\makeindex             

\usepackage{amsfonts}
\usepackage{amsmath}

\DeclareMathOperator{\R}{\mathbb{R}} 
\DeclareMathOperator{\N}{\mathbb{N}} 
\DeclareMathOperator{\Z}{\mathbb{Z}} 
\DeclareMathOperator{\var}{Var} 
\DeclareMathOperator{\0}{\mathbf{0}} 
\DeclareMathOperator{\geo}{\gamma} 

\newcommand{\olle}{H\"{a}ggstr\"{o}m}
\newcommand{\Zz}{\mathbb{Z}^2}
\newcommand{\pt}{\mathbf{t}}
\newcommand{\EE}{\mathbb{E}} 
\newcommand{\B}{\mathbf{B}} 
\newcommand{\T}{\mathcal{T}} 
\newcommand{\vgw}{\gamma(v,w)}
\newcommand{\region}{\mathcal{V}}
\newcommand{\Dir}{\mathrm{Dir}}
\newcommand{\slope}{x_n^{\lambda}}
\newcommand{\cyl}{\text{Cyl}}
\newcommand{\tn}{T(\0,(n,0))}
\newcommand{\strike}{Q}

\renewcommand{\P}{\mathbb{P}} 
\renewcommand{\E}{\mathcal{E}}

\begin{document}

\title*{Geodesic rays and exponents in ergodic planar first passage percolation.}
\author{Gerandy Brito and Christopher Hoffman}
\institute{Gerandy Brito \at Georgia Institute of Technology, College of Computing 801 Atlantic Drive. Atlanta, GA 30332-0280 USA, \email{gbrito3@gatech.edu}
\and Christopher Hoffman \at University of Washington, Department of Mathematics
Box 354350, C-138 Padelford. Seattle, WA 98195-4350 USA, \email{hoffman@math.washington.edu}}
%
%
\maketitle

 \abstract*{We study first passage percolation on the plane for a family of invariant, ergodic measures 
 on $\Z^2$. We prove that  for all of these models the asymptotic shape is the $\ell_1$ ball and that there are exactly four infinite geodesics starting at the origin a.s. In addition we determine the exponents for the  
 variance and wandering of finite geodesics. We show that the variance and wandering exponents do not satisfy the relationship of $\chi= 2\xi -1$ which is expected for independent first passage percolation.}

\abstract{We study first passage percolation on the plane for a family of invariant, ergodic measures 
 on $\Z^2$. We prove that  for all of these models the asymptotic shape is the $\ell_1$ ball and that there are exactly four infinite geodesics starting at the origin a.s. In addition we determine the exponents for the  
 variance and wandering of finite geodesics. We show that the variance and wandering exponents do not satisfy the relationship of $\chi= 2\xi -1$ which is expected for independent first passage percolation.}

\section{Introduction} \label{sec:motivation}

First passage percolation is a widely studied model in statistical physics. One of the
main reasons for interest in first passage percolation is that it is believed that, for independence passage times (and under mild assumptions on the common distribution) the model belongs to the KPZ universality class.
The study of first passage percolation has centered on the three main sets of questions below. 
(Precise definitions are given in the next two sections.)
\begin{enumerate}
\item {\bf Asymptotic shape}. Cox and Durrett proved that every model of first passage percolation has an asymptotic shape $\mathbf{B} \subset \R^2$ which is convex and has the symmetries of $\Z^2$ \cite{CD81}. We would like to determine $\mathbf{B}$ or at least describe some of its properties. In particular is the asymptotic shape is strictly convex and is its boundary differentiable?
\item {\bf Infinite geodesics from the origin}. Are there infinitely many one-sided infinite geodesics that start at 
$(0,0)$? Do these geodesics all have asymptotic directions?
\item {\bf Variance and wandering exponents}. For any $\lambda \geq 0$ does there exist a variance exponent  $\chi=\chi(\lambda)$ such that
\begin{equation*}
\var(T(0,(n,\lambda n))=n^{2\chi+o(1)}?
\end{equation*} Does there exist a wandering exponent $\xi=\xi(\lambda)$ such that with high probability 
every edge in $\geo(\0,(n,\lambda n))$ is within distance $n^{\xi+o(1)}$ of the line segment connecting
$\0$ and $(n,\lambda n)$? Do $\chi$ and $\xi$ satisfy the universal scaling relation
\begin{equation*}
\chi= 2 \xi -1?
\end{equation*}
\end{enumerate}

It is widely believed that
(under mild assumptions) in independent first passage percolation the answer to all of these questions is yes.
However in our models we show that the answer all of these questions is at least somewhat different than the answers that are expected for the independent case. Thus our model shows that universality cannot 
be expected to hold for all models of ergodic first passage percolation.
Our results are as follows.

\begin{enumerate}
\item For all of our models the asymptotic shape $\mathbf{B}$ is the unit ball in the $\ell_1$--norm. 
\item Our models have exactly four one-sided infinite geodesics starting from the origin a.s., each of which meander through a quadrant.
\item For each value of $\lambda$ we calculate exact variance and wandering exponents of the geodesic
from $\0$ to $(n,\lambda n)$. 
For all $\lambda>0$ the
variance exponent $\chi$ is zero while the wandering exponent is 1. For $\lambda=0$ we get variance and wandering exponents that satisfy $0<\chi=\xi <1$. In neither of these cases do the exponents satisfy the universal scaling relation
$\chi=2\xi -1$.
\end{enumerate}

It is already known that there exist models of ergodic first passage percolation whose behavior is different
from what is expected for independent first passage percolation. \olle \text{}
and Meester showed that for any set $\mathbf{B}\subset \R^2$ which is bounded, convex and has non-empty interior and all the symmetries of 
$\Z^2$ there is a model of ergodic first passage percolation that has $\mathbf{B}$ as its limiting shape \cite{hagmee95}. The examples
we construct show that the there are models of ergodic first passage percolation that
have anomalous geodesic structures. More interestingly our models have  anomalous variance and wandering exponents and these exponents depend on the direction. We are not aware of any other non-trivial models of 
ergodic first passage percolation where the variance and wandering exponents have been explicitly calculated.

\section{Background on first passage percolation.} \label{sec:intro}
In first passage percolation, a nonnegative variable is associated to each edge of a given graph. These variables give rise to a random metric space. Among the fundamental objects of study of this metric space are the scaling properties of balls and the structure of geodesics. By planar first passage percolation, we refer to the model on the lattice, denoted by $\Zz$, which has vertex set $V=\{(x,y)~ x,y\in \Z\}$ and edge set $\E=\{(v,w): |v-w|=1 \}\subset V\times V$ where $|\cdot|$ denotes the $\ell_1$ distance. A configuration of $\Zz$ is simply a function from the edge set to the nonnegative real numbers:
\begin{equation}
\pt: \E\rightarrow [0,\infty).
\end{equation}

We will use the more common notation $\pt_e$ for $\pt(e)$. If $\nu$ is a probability measure on $[0,\infty)^{\E}$, we denote the probability space with state space $[0,\infty)^{\E}$ and measure $\nu$ by $FPP(\nu)$ . 
The number $\pt_e$ can be seen as the passage time or length of the edge $e$. Given a configuration $\pt$ on $\Zz$ and a path $\pi=\{e_i\}_{i=1}^{k}$ the length of $\pi$ is $$\Gamma(\pi)=\sum_{i=1}^k \pt_{e_i}.$$
The distance between two vertices $u$ and $v$ is denoted by $T(u,v)$ and it is defined as
\begin{equation}\label{def:dist}
T(u,v)=\inf \Gamma(\pi)
\end{equation}
where the $\inf$ is taken over the set of all paths connecting $u$ and $v$. It is not hard to check that $(\Zz,T(\cdot,\cdot))$ is a pseudometric space for any configuration. Furthermore, if the values $\pt_e$ are all bigger than zero then $T(\cdot, \cdot)$ is a metric. As we will see in section \ref{sec:odometer}, our measures will be bounded away from zero, so for the rest of the paper $(\Zz,T(\cdot,\cdot))$ is a metric space. The ball of radius $R$ centered at $u$ is
\begin{equation}
B(u,R)=\{v\in V:~T(u,v) < R\}.
\end{equation}

Cox and Durrett \cite{CD81} studied the behavior of large balls after scaling. They proved that, if $\pt_e\sim \nu$ satisfying
\begin{equation}
\EE(\min\{\pt_1^2,\pt_2^2,\pt_3^2,\pt_4^2\})<\infty
\end{equation}
for independent copies of $\pt_e$, and the mass at zero is less than the threshold for bond percolation then there is a non-empty set $\mathbf{B}$, compact, convex and symmetric with respect to the origin such that, for any $\epsilon>0$
\begin{equation}
\P \left((1-\epsilon)\mathbf{B}\subset \frac{B(\0,R)}{R}\subset (1+\epsilon)\mathbf{B}~ \mbox{for all large} ~R\right)=1.
\end{equation}
Boivin extended this to a wide class of ergodic models of first passage percolation \cite{boivin90}.

The question of which compact sets can be obtained as limit in FPP is almost entirely open for the i.i.d. case. Interestingly, when we consider the bigger set of stationary and ergodic measures on $(\mathbb{R_+})^{\E}$ it was proved by Haggstrom and Meester \cite{HM95} that any compact, convex, symmetric (with respect to the origin) set is the limiting shape for a stationary and ergodic measure, not necessarily i.i.d. It is worth mentioning that the limiting shape $\B$ is the unit ball of a norm $\|\cdot\|_{\nu}$. This norm can be computed as follows. First, we extend $T(\cdot, \cdot)$ to $\R^2$ by setting $T(u,v)=T(\lfloor u\rfloor, \lfloor v\rfloor)$ for all $u,v\in \R^2$. Here we slightly abused notation using the floor function on points. The reader should understand this by applying it to each coordinate. It can be shown, under mild assumptions on the distribution of $\pt_e$, that the norm $\|x\|_{\nu}$ satisfies

\begin{equation*}
\|x\|_{\nu}=\lim_{n}\frac{T(\0,nx)}{n}
\end{equation*}
where the limit exists a.s. and in $L_1$ for every fixed $x\in \R^2$. The set $\B$
\begin{equation}\label{normB}
\B=\{x\in \R^2:~\|x\|_{\nu}\leq 1~\}.
\end{equation}

A \textit{geodesic} between $u$ and $v$ is a path that realizes the infimum in \eqref{def:dist}. We denote geodesics by $\gamma(u,v)$. Geodesics aren't always unique. A simple condition to guarantee such property for independent edge weights is to consider continuous distribution for $\pt_e$. A \textit{geodesic ray} is an infinite path $\{v_0,v_1,v_2, \dots\}$ such that every finite sub-path is a geodesic between its endpoints. We consider two geodesic rays to be distinct if they intersect in only finitely many edges. We denote by $\T_0$ the set of all geodesic rays starting at the origin. Ahlberg and Hoffman \cite{AH16} recently showed that for a wide class of measures the cardinality of $\T_0$ is constant almost surely, possibly $\infty$.

\section{Statement of Results}\label{results}

The limiting shape $\B$ is closely related to the number and geometry of geodesic rays for ergodic FPP. Let $sides(\B)$ denote the number of sizes of $\B$ if it is a polygon, and infinity otherwise, Hoffman (\cite{hoffman05}, \cite{H08}) proved that, for any $k\leq sides(\B)$ there exist $k$ geodesic rays almost surely, for \textit{good measures}, see Section \ref{sec:odometer} for details.  In particular, his results imply that there exist at least four geodesics a.s. When $\B$ is a polygon, little is known about existence of geodesics rays in the direction of the corners of $\B$. Recently, Alexander and Berger \cite{aleber18} exhibit a model for which the limiting shape is an octagon and all (possibly infinitely many) geodesic rays are directed along the coordinate axis. Our first result shows that our model has exactly four geodesic rays a.s.. To the best of our knowledge, this is the first known FPP model for which $|\T_0|$ is finite.

\begin{theorem}\label{four_rays}
There exists a family of measures $\{\nu_{\alpha}\}_{0<\alpha<0.2}$ such that $|\T_0|=4$ $\nu_{\alpha}$-almost surely.
\end{theorem}

Our next result is about the direction of geodesic rays. We start with a definition. The \textit{direction}, $Dir(\Pi)$, of a sequence of (not necessarily disticnt) points $\Pi=\{v_k,~k\geq 0\}$ is the set of limits of $\{v_k/|v_k|,~v_k\in \Pi\}$. If $|v_k| \to \infty$ then $Dir(\Pi)$ is a connected subset of $S^1$. Damron and Hanson \cite{DH14} were the first to prove directional results for geodesic rays for \textit{good measures}  that also have the upward finite  energy property. Their results are also dependent on the geometry of $\B$ in the following way. We say that a linear functional $\rho:\mathbb{R}^2\rightarrow \mathbb{R}$ is tangent to $\B$ if the line $\{x\in \mathbb{R}^2:~\rho(x)=1\}$ is tangent to $\B$ at 
a point of differentiability of the boundary of $\B$. In view of equation \eqref{normB}, we can write the intersection of this tangent line and the boundary of $\B$ as 
a set in $S^1$:
\begin{align}\label{def: gen_dir}
D_{\rho}=\{x\in S^1:~ \rho(x)=\|x\|_{\nu} \}.
\end{align}

\cite[Theorem 1.1]{DH14} states that for any functional $\rho$, tangent to $\B$, there is an element $\gamma\in \T_0$ satisfying $Dir(\gamma)\subset D_{\rho}$. Because of the differentiability condition, their result gives no information about the behavior around corners of $\B$. For our family of measures $\{\nu_{\alpha}\}$ we are able to completely characterize the directions of geodesic rays. 

\begin{theorem}\label{dir_rays}
Fix $\alpha$ such that $0<\alpha<0.2$ and consider  $FPP(\nu_{\alpha})$.
Let $\rho$ be a linear functional tangent to the $\ell_1$-ball. There is exactly one geodesic with generalized direction equal to $D_{\rho}$.
\end{theorem}

Lastly, we turn our attention to the geometry of finite geodesics. We follow the classical approach and study it both from the random and the geometric point of view. For the first one, the most basic analysis comes from understanding the variance of $T(\0, x)$. As stated informally in the introduction, it is believed that there exist a universal exponent that governs this quantity. For our model, we show the existence of a constant $\chi = \chi(\lambda)$ and a universal constant $K$ such that 
\begin{equation*}
\frac{1}{K}\leq \frac{Var (T(\0,(n,\lambda n)))}{n^{2\chi}}\leq K,
\end{equation*}
see Lemma   \ref{naomi} and Lemma \ref{thief} for the formal statements. We point out that our results are strong enough that they satisfy any reasonable definition of  the variance exponent, in particular those suggested in \cite{ADH17}. 

From the geometric perspective, we look at how far is $\gamma(\0,(n,\lambda n))$ from the straight line connecting the origin and the point $(n,\lambda n)$. As in the case of the variance exponent, it is widely believed that there is a universal constant, denoted by $\xi$, such that the maximal distance between  $\gamma(\0,(n,\lambda n))$ and the line through the origin and  $(n,\lambda n)$ is of order $n^{\xi}$. So, if one sets $\cyl(\slope, N)$ to be the set of points within (Euclidean) distance $N$ from the line connecting $\0$ and $\slope=(n,\lambda n)$, it is expected that $N=n^{\xi}$ is the right scale of cylinder to contain $\gamma(\0,\slope)$. To formally capture this property we adopt the definition in \cite{NP95} which we reproduce below.

\begin{definition}\label{def:xi}
For $0\leq \lambda \leq 1$ and $\slope=(n,\lambda n)$ we set
\begin{equation*}
\xi_{\lambda}=\inf\{\alpha>0:~ \liminf_n~\P(\gamma(\0,\slope)\subset \cyl(\slope,n^{\alpha}))>0\}.
\end{equation*} 
\end{definition}

We compute this exponent in any direction, and confirm non-universality of FPP for invariant, but not necessarily identically distributed, edge weight.

\begin{theorem}\label{Thm3}
Fix $\alpha$ such that $0<\alpha<0.2$ and consider  $FPP(\nu_{\alpha})$.
In every direction not parallel to the coordinate axes we have the variance exponent $\chi= 0$ and
the wandering exponent $\xi=1$. Parallel to the coordinate axes the two exponents are equal with
$\chi =\xi = \frac{\log 5}{\log 5 -\log \alpha}$. In no direction do the exponents satisfy the universal scaling
relation.
\end{theorem}

\begin{remark}
This is the content of Lemmas \ref{thief} and \ref{coaching} for the coordinate directions
and Lemma \ref{naomi} and Propositions \ref{osaka} and \ref{prop:exp3} for the non-coordinate directions. 
The reader will notice that these results lead to stronger statements. In particular, we can deduce that, for $0<\lambda\leq 1$
\begin{equation*}
\lim_n ~\P(\gamma(\0,\slope)\subset \cyl(\slope,n^{1+\epsilon}))=1
\end{equation*} 
and
\begin{equation*}
\lim_n ~\P(\gamma(\0,\slope)\subset \cyl(\slope,n^{1-\epsilon}))=0
\end{equation*}  
and similarly for the coordinate directions and $\xi=0$. In particular, we believe that for our model the results of Theorem \ref{Thm3} apply to any reasonable definition of the variance and wandering exponents.

\end{remark}

\subsection{Organization of the paper}

The rest of the paper is organized as follows. In Section \ref{sec:odometer} we define the measures $\nu_{\alpha}$ and show its main properties. The proof of Theorem \ref{four_rays} is given in Section \ref{sec:det} where the limiting shape is also determined. Section \ref{direction} is devoted to the study of the directional properties of geodesic rays and the proof of Theorem \ref{dir_rays}. 
In Sections \ref{sec:var} and \ref{sec:cor} we prove our final theorem which determines the exponents
in all directions. This is the content of Lemmas \ref{thief} and \ref{coaching} for the coordinate directions
and Lemma \ref{naomi} and Propositions \ref{osaka} and \ref{prop:exp3} for the non-coordinate directions.

\section{Construction of the measures $\{\nu_{\alpha}\}$.}\label{sec:odometer} 

In this section we construct a family of measures $\{\nu_{\alpha}\}\subset \mathcal{M}((\mathbb{R}_+)^{\E})$, indexed by a parameter $0<\alpha<0.2$. We state their main properties and study the behavior of geodesics for $FPP(\nu_{\alpha})$. 

Let $\Omega =\{0,1,2,3,4\}^{\N}$. Let $\sigma: \Omega \to \Omega$ be the 5-adic adding machine: $\sigma$ adds one (mod 5) to the first coordinate. If the result is not zero then we leave all the other coordinates unchanged. If the result is zero then we add one to the second coordinate. We repeat until we get the first non-zero coordinate. All subsequent coordinates are left unchanged. Thus 
\begin{equation*}
\sigma(0,1,2,\dots)=(1,1,2,\dots) \ \ \ \ \text{and} \ \ \ \ \sigma(4,4,2,1,\dots)=(0,0,3,1,\dots).
\end{equation*}
We also adopt the convention $\sigma(444\dots)=000\dots$

It follows that $\sigma: \Omega \to \Omega$ is uniquely ergodic with respect to the uniform measure on $\Omega$.
We use this map to form a $\Z^2$ action of $\Omega \times \Omega$ as follows. Let $L: \Omega \to \N$ given by
\begin{equation}\label{k-grid1}
L(\omega)=\min \{i:\  \omega(i)>0\}.
\end{equation}
For $(\omega_1,\omega_2) \in \Omega \times \Omega$ fixed and $v=(x(v),y(v)) \in \Zz$ we define:
\begin{equation}\label{k-grid2}
k(v,v+\mathsf{e}_1)=
L(\sigma^{y(v)}(\omega_1))
\end{equation}
and
\begin{equation*}
k(v,v+\mathsf{e}_2)=
L(\sigma^{x(v)}(\omega_2))
\end{equation*}
where $\mathsf{e}_1, \mathsf{e}_2$ are the vectors in the canonical base of $\R^2$. The following set will be referred to often in the paper, so we highlight its definition now. 
\begin{definition}\label{def:grid}
We denote the set of edges $e \in \E$ such that $k(e) \geq j$ as the $j-$grid.
\end{definition}

It is helpful to visualize the $j-$grid. Note that, by definition, these subgraphs are nested:

\begin{equation*}
1\mbox{-grid} \supseteq 2\mbox{-grid} \supseteq \dots 
\end{equation*}
Also, since edges $e$ and $e'$ in the same horizontal or vertical line satisfy $k(e)=k(e')$, it is not hard to see that for each $j$, the subgraph induced by the $j$-grid is isomorphic to $5^j\Z^2$. 

We are ready to define the measure $\nu_{\alpha}$. Let $\{X_{j,e}\}_{j \in \N, e \in \E}$ be a set of independent random variables where
$X_{j,*}$ has the uniform distribution over the interval $[0,\frac{\alpha^j(1-\alpha)}{1000}]$. We take 
\begin{equation*}
\pt_e=1+\alpha^{k(e)}+X_{k(e),e}
\end{equation*} 
where $\omega_1$ and $\omega_2$ are chosen uniformly i.i.d.  and independent of
the $\{X_{j,e}\}$. For technical reasons, we let $0<\alpha< 0.2$. 
Note that for every edge $e$
\begin{equation} \label{circulo}
1 \leq \pt_e \leq 1.3.
\end{equation}
We set $\nu_{\alpha}$ to be the resulting measure on $(\mathbb{R}_+)^{\E}$.

\begin{remark}
The measures $\nu_{\alpha}$ fall into the class of \textit{good} measures introduced in \cite{hoffman05} and \cite{H08}. 

We recall the definition of good measures. A measure $\P$ is good if:
\begin{itemize}
\item[(a)] $\P$ is ergodic with respect to the translations of $\Zz$.
\item[(b)] $\P$ has all the symmetries of $\Zz$.
\item[(c)] $\P$ has unique passage times.
\item[(d)] The distribution of $\P$ on an edge has finite $2+\epsilon$ moment.
\item[(e)] The limiting shape is bounded.
\end{itemize}
The construction of $\nu_{\alpha}$ is done so properties $(a)-(e)$ are easy to check.
\end{remark}

Informally, we think of a realization of $FFP(\nu_{\alpha})$ as building a series of horizontal and vertical highways on the nearest neighbor graph of $\Z^2$. 
The value of $\omega_i,~i=1,2$, determines where the origin lies with respect to these highways.
By construction, edges in the $j$-grid are faster (i.e.: have smaller passage time) than edges in any $j'$-grid  for $j>j'$. Hence, a geodesics ray is expected to follow one grid until it encounters a faster one. Then the geodesic continues along edges of the faster grid. Globally, we expect to see rays with longer segments parallel to the axes as they move away from the origin. We also suspect that the length of these horizontal or vertical segments is roughly determined by the value of the $j$-grid they are part of. We formalize this intuition in the next sections. 

\section{Structure of finite geodesics} \label{sec:det}

In this section we present several properties of geodesics in $FPP(\nu_{\alpha})$. The first lemmas describe the geometric properties of finite geodesics along vertices in the $k-$grid, recall Definition \ref{def:grid}. 

\begin{lemma}\label{square}
Let $C=C(x,y,k)$ a square of side $5^k$ with lower left vertex $(x,y)$ such that all the edges in its boundary are in the $k-$grid. Consider two vertices $v$ and $w$ in the boundary of $C$. Then
\begin{itemize}
\item[(i)] $\vgw$ is completely contained in $C$.
\item[(ii)] If $v$ and $w$ lie in the same or adjacent sides of $C$, $\vgw$ lies in the boundary of $C$.
\end{itemize}
\end{lemma}

\begin{proof}
We argue by contradiction. Assume there are vertices $v$ and $w$ in the boundary of $C$ such that $\vgw$ intersects the complement of $C$. Because a subpath of a geodesic is also a geodesic, we can assume that the edges of $\vgw$ lie entirely in the complement of $C$, by considering a segment of $\vgw$ completely in the complement of $C$ and taking $v$ and $w$ to be its end points. 
Let $d$ denote the length (the $\ell_1-$distance) of the shortest path along the boundary of $C$ connecting $v$ and $w$. 
 
If the maximal distance from a vertex in $\vgw$ to $C$ is less than $5^k$ then, by construction, all edges in $\vgw$ will lie on the $(k-1)-$grid at most, hence,  have passage time at least $1+\alpha^{k-1}$. Then the passage time of $\vgw$ is at least
\begin{equation*}
d(1+\alpha^{k-1})> d(1+\alpha^k)+d\alpha^k
\end{equation*}
The right hand side is an upper bound for the passage time of the path from $v$ to $w$ along the boundary of $C$. We conclude that going along the boundary of $C$ will be a shortest path from $v$ to $w$. Hence, there should be a vertex in $\vgw$ at distance at least $5^k$ of $C$. Then the passage time of $\vgw$ is at least
\begin{equation*}
2(5^k)+d
\end{equation*}
where the factor of two appears since we move away from $C$ at least $5^k$ edges and come back to $C$, crossing another $5^k$ edges. Observe that $d\leq 2(5^k)$. Hence,
\begin{equation*}
d(1+\alpha^k)+d\alpha^k=d+2d\alpha^k< d+2(5^k),
\end{equation*}
using that $2\alpha^k< 1$ as long as $\alpha<1/5$. The left hand side above is an upper bound on the passage time of a path connecting $v$ and $w$ along the boundary of $C$. This concludes the proof of part $(i)$.\\
To prove $(ii)$, assume that $v$ lies in the left side of $C$ and consider two cases for $w$.\\

{\bf Case 1: $w$ lies also on the left hand side or the horizontal sides of $C$, but it is not a corner on the right hand side.} By (i) we know $\vgw$ is contained in $C$. If $\vgw$ uses edges in the interior of $C$, we can assume, changing $v$ and $w$ if necessary, that the entire geodesic lies in the interior. This implies that all edges in $\vgw$ have passage times at least $(1+\alpha^{k-1})>1+\alpha^k+\frac{\alpha^k(1-\alpha)}{1000}\geq t_e$ for all $e$ in the boundary of $C$.  Hence, a path along the boundary will have smaller passage times, which shows that $\vgw$ lies on the boundary.

{\bf Case (2): $w$ is a corner on the right hand side of $C$.} We compare the path along the boundary of $C$ to any path $\pi$ which traverses edges in the interior of $C$. Observe that $\pi$ will have to traverse at least $5^k$ many edges horizontally, because $v$ and $w$ are on opposite sizes of $C$, and at least |y(w)-y(v)| many edges vertically. This leads to the lower bound
\begin{equation*}
\Gamma(\pi)\geq 5^k(1+\alpha^{k-1})+|y(w)-y(v)|.
\end{equation*}
A path on the boundary of $C$ connecting $v$ and $w$ has length at most:
\begin{equation*}
5^k(1+\alpha^{k})+|y(w)-y(v)|(1+\alpha^{k})+2(5)^k\frac{\alpha^k(1-\alpha)}{1000}.
\end{equation*}
The last summand is an upper bound on the sum of the random portion of the path's distance. To conclude it suffices to show that
\begin{equation*}
5^k(1+\alpha^{k-1})+|y(w)-y(v)|\geq 5^k(1+\alpha^{k})+|y(w)-y(v)|(1+\alpha^{k})+2(5)^k\frac{\alpha^k(1-\alpha)}{1000}.
\end{equation*}
This inequality is equivalent to
\begin{equation*}
2(5)^k\left(1-\alpha-\frac{\alpha(1-\alpha)}{1000}\right)\geq |y(w)-y(v)|\alpha
\end{equation*}
which follows directly since $5^k\geq |y(w)-y(v)|$ and $1-\alpha-\frac{\alpha^k(1-\alpha)}{1000}\geq \alpha$ for $0<\alpha<\frac{1}{5}$.\\
\end{proof}

\begin{corollary}\label{corner}
In the setting of Lemma \ref{square}, let $v$ and $w$ be \textit{any} vertices in the boundary. Assume that $\vgw$ visits a corner of $C$. Then $\vgw$ is completely contained in the boundary of $C$.  
\end{corollary}
\begin{proof}
Let $v'$ be a vertex in $\vgw$ which is in the corner of $C$. Then $v'$ is in two sides and both the other two sides are adjacent to one of these two sides. Then both the pairs $v$ and $v'$ and $v'$ and $w$ lie in (the same or) adjacent sides of $C$. Thus the corollary follows from Lemma \ref{square} applied to $\gamma(v,v')$ and $\gamma(v',w)$.
\end{proof}
We extend the result above to a large rectangle in the next lemma.

\begin{lemma}\label{rectangle}
Let $M=M(\alpha,k)$ be an integer such that $\alpha^kM>1$, and let $R=R(x,y,k)$ be a rectangle with vertices $(x,y); (x,y+5^k); (x+5^kM,y); (x+5^kM,y+5^k)$ such that all the edges in its sides are in the $k$-grid.
Let $v$ and $w$ be vertices in the boundary of $R$ such that at least one is on one of the shorter sides of $R$. If $\vgw$ is contained in $R$, then it is contained in the $k\mbox{-grid}$.
\end{lemma}

\begin{remark}\label{remark_rectangle}
The lemma above is still true if the largest side of the rectangle is parallel to the $y$-axis.
\end{remark}

\begin{remark}
The lemma above confirms that, once a geodesics enters a fast grid, it will not visit slower edges anymore: the only edges parallel to the $x$-axis in $\vgw$ are in the boundary of $R$. It may traverse edges in the interior of $R$ but only parallel to the $y-$axis and on the $k-$grid.
\end{remark}
\begin{proof}
To fix ideas, assume $v$ lies on the left side of $R$. Notice that $R$ can be divided into $M$ squares of side $5^k$, each satisfying the condition of Lemma \ref{square}, namely, each has boundary edges in the $k$-grid. We denote these squares by $C_1,~ C_2\dots, C_M$ from left to right. Also, for $1\leq j \leq M$ denote $v_j$ and $w_j$ the first and last vertex that $\vgw$ visits in $C_j$, respectively, when this intersection is not empty. We will again split the proof into cases.\\
{\bf Case 1: $w$ lies on the common boundary of $C_1$ and $R$.}  This is the content of Lemma \ref{square}.\\
{\bf Case 2: $w$ lies in one of the larger (horizontal) sides of $R$.} This case follows by induction on $M$, with case 1, or Lemma \ref{square},  being the initial step. Note that in this case we may traverse edges in the interior of $R$, but only in the boundary of $C_j$ for some values of $j$, thus we are still on the $k-$grid.\\
{\bf Case 3: $w$ is in the right side of $R$.}  Assume $\vgw$ visits a vertex on the horizontal sides of $R$, say $v'$. Then by {\bf case 2} applied to $\gamma(v,v')$ and $\gamma(w,v')$ we conclude that $\vgw$ is on the union of the boundaries of the $C_j$, which is a subset of the $k-$grid. If such vertex $v'$ does not exist, we deduce that all horizontal edges (i.e.: parallel to the $y-$axis) on $\vgw$ are in the interior of $R$. Then its length will be at least:
\begin{equation*}
5^kM(1+\alpha^{k-1}),
\end{equation*}
since all edges in the interior are in the $(k-1)-$grid at most. The shortest path on the boundary of $R$ connecting $v$ and $w$ has length bounded above by
\[
5^k(M+1)(1+\alpha^{k}+\frac{\alpha^k(1-\alpha)}{1000}).
\]
It can be checked that, for our choice of $M$ it holds $5^kM(1+\alpha^{k-1})\geq 5^k(M+1)(1+\alpha^{k})+5^k(M+1)(\alpha^k\frac{1-\alpha}{1000})$, which yields the desired contradiction. We have proved that $\vgw$ lies on the union of the boundaries of $C_j$, which proves the lemma.  
\end{proof}

\begin{proposition}\label{samevertex}
Let $v, w$ vertices that are end points of edges in the $k$-grid, satisfying $|v-w|\geq 5^{2k}M$, for an integer $M$ such that $\alpha^kM>1$. Then the geodesic $\vgw$ is contained in the $k$-grid.
\end{proposition}
\begin{proof}
Suppose that $\vgw=\{e_1,e_2,\dots, e_t\}$ contains at least one edge outside the $k-$grid. Let
\[
m=\min\{s\geq 1:~ e_s~ \mbox{ is not on the $k$-grid}\}.
\]
By definition, one endpoint of $e_m$ lies in the $k$-grid and one endpoint lies outside of it.
 Let $R^{e_m}$ be the unique rectangle defined by the following constraints:
 \begin{itemize}
 \item[(a)]  The boundary of $R^{e_m}$ is a subset of the $k-$grid. Furthermore, the length of the sizes of $R^{e_m}$ are $5^k$ and $5^kM$ for a natural number $M$ such that $\alpha^kM>1$.
 \item[(b)]  The boundary of $R^{e_m}$ intersects $e_m$.
 \item[(c)]  $R^{e_m}$ contains $e_m$.
 \item[(d)]  The larger sides of $R^{e_m}$ are parallel to $e_m$.
 \end{itemize}
 
 Note that the conditions in $(a)$ are as those in Lemma \ref{rectangle}. Since $e_m$ is the first edge on $\vgw$ not on the $k-$grid, exactly one endpoint of it is in this grid, which is the only possible intersection in condition $(b)$. We will denote this vertex by $v'$. Finally, all of the conditions  ensure that we have a unique choice of $R^{e_m}$.
 
 \begin{figure}[h]
\includegraphics[scale=.4]{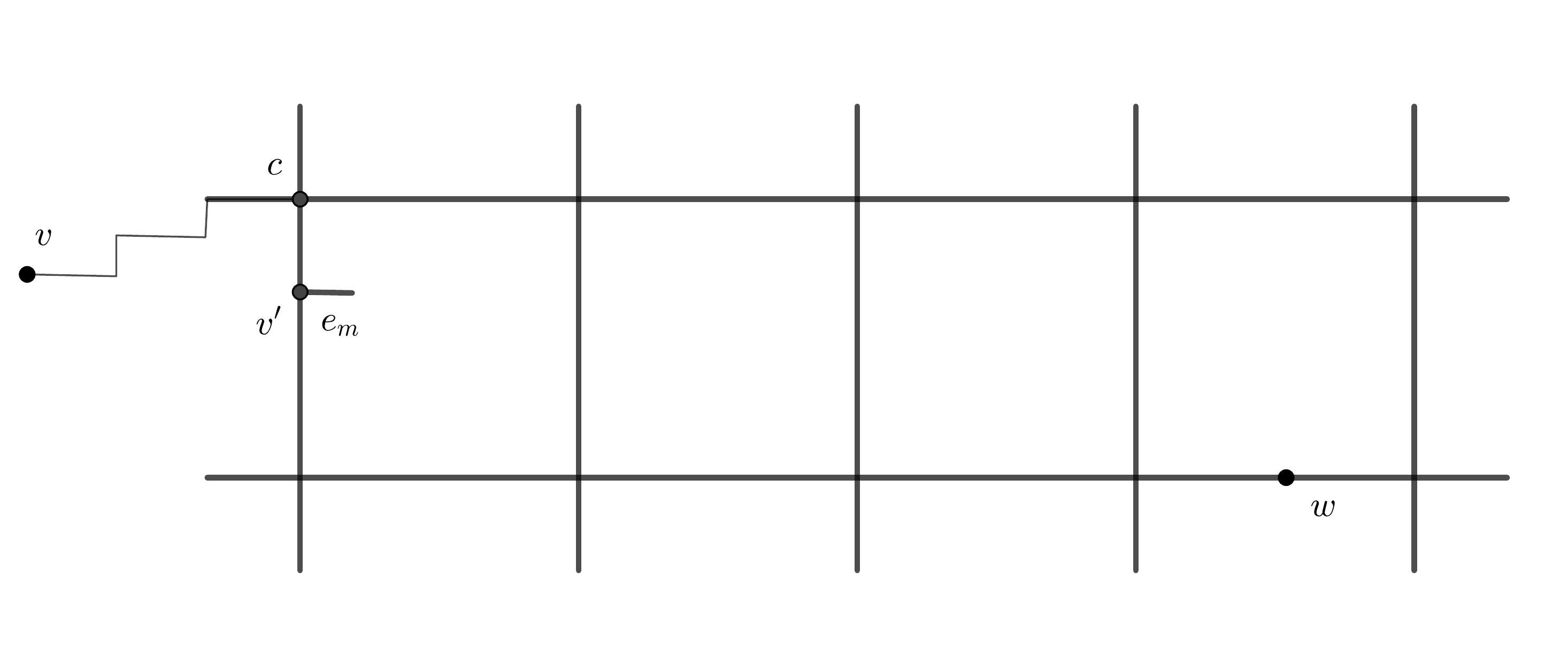}

\caption{The rectangle $R^{e_m}$ in {\bf case 2} of the proof of proposition 1: the geodesic will connect vertices $v$ and $v'$ traversing edges in the $k-$grid only, which forces it to visit a corner of $R^{e_m}$, denoted by $c$. Then the geodesics follows from $c$ to $w$ traversing $e_m$, which is not the fastest path, by Lemma \ref{rectangle}.}
\label{fig_rectangle}       
\end{figure}
 
We consider two cases:\\
\textbf{Case 1:~ $w$ is in the complement of $R^{e_m}$.} Then in order to reach $w$, $\vgw$ has to exit $R^{e_m}$ for the last time at some vertex $w'$ in its boundary. Since one endpoint of $e_m$ is inside $R^{e_m}$ we have that $w'\neq v'$.  By Lemma \ref{rectangle} or Remark \ref{remark_rectangle} the geodesic from $v'$ to $w'$ is contained in the $k-$grid and hence it cannot traverse $e_m$. This contradicts our assumption.\\
\textbf{Case 2: ~ $w$ is in $R^{e_m}$.} We start making two simple observations. First, because $e_m$ is parallel to the longer sides of $R^{e_m}$, we deduce that $v'$ is on one of its shorter sides. Second, from the assumption that $v$ and $w$ are far away from each other, we conclude that $v$ is in the complement of $R^{e_m}$. By definition of $m$, all edges $e_s,~s<m$ are in the $k$-grid. Traverse $\vgw$ from $v$ until we get to $v'$. We claim that we must visit one of the corners of the side of $R^{e_m}$ that contains $v'$. To see this, simply observe that removing the two corners of such side disconnects $v'$ from $v$ in the graph induced by the $k-$grid (because $v$ is in the complement of $R^{e_m}$). Call the visited corner $c$ and let $C$ be the square of size length $5^k$ completely inside $R^{e_m}$ with one corner equal to $c$. Let $\gamma'$ be the intersection of $C$ and $\gamma(c,w)$. Note that $e_m\in \gamma'$. Now $\gamma'$ is connected, and its endvertices are $c$, a corner of $C$, and one vertex on its boundary. In the square $C$, because $c$ is a corner, it lies on the same or an adjacent side of such vertex, and thus by Lemma \ref{square} $(ii)$ we get that $\gamma'$ is completely on the boundary of $C$, which contradicts the definition of $e_m$ and finishes the proof.

\end{proof}

To prepare the ground for our next lemma, we draw a few conclusions from Proposition \ref{samevertex}. First, notice that any geodesic ray $\gamma$ will have infinitely many vertices in the $k$-grid, for all $k$. If $v_k$ is the first such vertex, it follows that \textit{all} edges in $\gamma$ after $v_k$ are in the corresponding grid. Applying the same reasoning we conclude that the intersection of $\gamma$ and the $k$-grid is an infinite connected set. The vertices $v_{k}$s  break $\gamma$ into slower edges, those with passage time of order $1+\alpha^k$, and faster edges, with passage time of order $1+\alpha^{k+1}$. We turn our attention to a set of special vertices and introduce the following definition.

\begin{definition}\label{loop}
 Let $v$ be a vertex of $\Zz$. We denote by $\region_k(v)$ the square in the plane containing $v$ with the following properties: 
 \begin{itemize}
 \item[(a)] The boundary of  $\region_k(v)$ is in the $k-$grid.
 \item[(b)] The area of $\region_k(v)$ is $5^{2k}$.
 \item[(c)] If $v$ lies in the intersection of two or more such squares, $\region_k(v)$ is the only one to the right and/or above $v$.
 \end{itemize}
 Denote by $v^k_i(v),~1\leq i\leq 4$ the corners of $\region_k(v)$, starting at the upper right and going counterclockwise. 
\end{definition}

Conditions $(a)$ and $(b)$ imply that from all bounded regions in the plane with $v$ in its interior and boundary a subset of the $k-$grid, $\region_k(v)$ is the one with smaller area. Condition $(c)$ handles the case when $v$ is in the boundary of such region. When there is no confusion, we will drop the dependence on $v$ in $\region_k$ and $v^k_i$. The importance of these vertices is explained in the next lemma.  

\begin{lemma}\label{inf_geo}
Let $\gamma$ be a geodesic ray starting at $v$. For each $k$, there is at least one value $1\leq i\leq 4$ such that $v^k_i\in \gamma$.
\end{lemma}

\begin{proof}
Consider $w\in \gamma$ be a vertex in the $k$-grid such that $d(v,w)\geq 5^{3k}M$, for $M$ which was defined in the proof of Lemma \ref{rectangle}. The existence of $w$ can be deduced from the fact that the $k-$grid  is isomorphic to $5^k\Zz$ and then we can find infinite closed paths on it disconnecting $v$ from infinity. Since  $\gamma$ is an infinite path it will intersect the $k-$grid infinitely many times. Let $\hat{v}\in \region_k(v)$ denotes the first vertex in the $k$-grid that we encounter while going along $\gamma$, starting at $v$. We have $d(v,\hat{v})\leq 5^k$ and thus, by Proposition \ref{samevertex}, the subpath from $\hat{v}$ to $w$ is contained in the $k$-grid. Thus, the last vertex that $\gamma$ visits in $\region_k$ is one of its corners.
\end{proof}

We are ready to prove Theorem \ref{four_rays}.
\subsection{Proof of Theorem \ref{four_rays}.}

Assume there exits five different $\gamma_i\in \T_0$, $i\in \{1,2,3,4,5\}$. Then there is a (random) ball $B$ centered at $v$ sufficiently large such that any two of these five geodesic rays only intersect in the interior of $B$. Take $k$ large such that $\region_k=\region_k(\0)$ has its four corners in the complement of $B$. By Lemma \ref{corner} each $\gamma_i$ will visit at least one corner of $\region_k$. This contradicts the intersection property since 
$\region_k$ has four corners. This proves
\[
|\T_0|\leq 4
\]
Since $\nu_{\alpha}$ is good, it follows from \cite[Theorem 1.2]{H08} that $|\T_0|\geq 4$ and the result follows.

This result allows us to determine the shape. A direct proof of the shape is also very short.
\begin{corollary}\label{shape}
The limiting shape of $FPP(\nu_{\alpha})$ is the $\ell_1$--ball. 
\end{corollary}

\begin{proof}
From Theorem \ref{four_rays} and  \cite[Theorem 1.2]{H08} we have that
the limiting shape of $FPP(\nu_{\alpha})$ is either proportional to the $\ell_1$--ball 
or the $\ell_\infty$--ball. As every edge has passage time at least one the limiting shape must be
contained in the $\ell_1$--ball. But the speed in the coordinate directions is one so the 
limiting shape must be the $\ell_1$--ball. 
\end{proof}

\section{Direction of the geodesic rays and proof of Theorem \ref{dir_rays}.}\label{direction}

Our goal in this section is to completely characterize $Dir(\gamma)$ for each geodesic ray $\gamma$ in $FPP(\nu_{\alpha})$. We start by combining Theorem \ref{four_rays} and recent results of Ahlberg and Hoffman \cite{AH16} to get further information about the geodesic rays. 

Throughout  this section, we will write $\region_k$ to refer to $\region_k(\0)$. Similarly, the corners of $\region_k$ will be denoted by $\{v^k_i\}_{i=1}^k$, see definition \ref{loop}. For $1\leq i\leq 4$, denote by $C_i=\{v^k_i,~k\in \N\}$ the set of corners lying in the $i$th quadrant of the coordinate plane. 

For any geodesic $\gamma$ recall that $\Dir(\gamma) \subset S^1$. In the remainder of the section we will slightly abuse notation by considering $\Dir(\gamma) \subset [0,2\pi)$.

\begin{lemma} \label{dh lemma} With $\nu_{\alpha}$ probability one the following holds: for each $1\leq i\leq 4$ there is a unique geodesic ray $\gamma_i$, starting at the origin, such that the angle $(i-\frac12)\frac{\pi}{2} \in Dir(\gamma_i)$
and $Dir(\gamma_i)$ is in the $i$th quadrant. 
\end{lemma}

\begin{proof}
For each quadrant there is a linear function $\rho_i$ whose level set 
$D_{\rho_i}(z)=1$ (see equation \eqref{def: gen_dir} for the definition of $D_{\rho}$) is the intersection of the boundary of the $\ell_1-$ball with the $i$th quadrant.
By Theorems 1.11 and 4.6 of \cite{DH14} for each $i$ there is a 
geodesic whose Busemann function is asymptotically linear with growth rate $D_{\rho_i}$
and whose $Dir(\gamma_i)$ is contained in the $i$th quadrant.
As there are only four geodesics a.s.,\  these geodesics are unique.
We denote by $\gamma_i$ the only geodesic ray directed on the $i$th quadrant.

For any $v,w \in \Z^d$ we have for all $k$ sufficiently large that 
$\region_k(v)=\region_k(w)$. Thus by Lemma \ref{inf_geo} we have that the geodesics are coalescing.
Thus $Dir(\gamma_i)$ is an almost sure\ invariant subset of $S_i$.
Either 
\begin{equation*}
\P(Dir(\gamma_1) \cap [0,\pi/4] \neq \emptyset) \geq 0
\end{equation*}
or
\begin{equation*}
\P(Dir(\gamma_1) \cap [\pi/4,\pi/2] \neq \emptyset) \geq 0.
\end{equation*}
By symmetry they must both be greater than zero. By shift invariance they both must have probability
one. As $\Dir(\gamma_1)$ is connected subset of $[0,\pi/2]$ then $\pi/4 \in \Dir(\gamma_1)$. 
The same argument works for the other three quadrants.
\end{proof}

\begin{lemma}
With $\nu_{\alpha}$ probability one it holds that
$v^k_i \in \gamma_i$   
for $1\leq i\leq 4$ and for all but finitely many $k$. 
\end{lemma}

\begin{proof}
For each $k$ there exists an $i$ such that both coordinates of $v^k_i$ are at least $5^k/2$ in absolute value. For such values of $k$ and $i$, we have
$Dir(v^k_i) \in (i-1)\pi/2 +(.1, \pi/2-.1)$.
Let $K$ be large enough such that for each $k>K$ we have that  $v^k_i$ for each $i$ is in a distinct geodesic (the existence of such $K$ follows from Theorem \ref{four_rays} and Lemma \ref{inf_geo}).
Also, for any $i$ and any vertex $v \in \gamma_i$ such that $|v| \geq \min_j|v^k_j|$ we have $Dir(v) \in (i-1)\pi/2+(-.01,\pi/2+.01)$. 
Then for this particular $i$ we have that $v^k_i \in \gamma_i$. From this we can conclude that for all other  $j \neq i$
we have that $v^k_j \in \gamma_j$ as well.
\end{proof}

\begin{lemma}\label{origin_position}
Let $\omega_1, \omega_2\in \Omega$  be sampled uniformly i.i.d.. The position of the origin in the interior of $\region_k$ is completely determined by the first $k$ entries of $\omega_i$, $i=1,2$.
\end{lemma}

\begin{remark}
Lemma \ref{origin_position} can be interpreted as follows: a realization of $\omega_1, \omega_2$ determines which edges are in the $k-$grid via equations \eqref{k-grid1} and \eqref{k-grid2}, and thus determines $\region_k=\region_k(\0)$. This region is a square of size $5^k$ and by definition the origin is one of the $5^{2k}$ vertices in it that do not lie on the top or right sides (see definition \ref{loop}). The lemma above tells us that it is enough to know the first $k$ coordinates of $\omega_i,~i=1,2$ to determine which of those $5^2k$ vertices is the origin. Note that, to know the region $\region_k$ itself we need to know all coordinates of $\omega_i~i=1,2$. 
\end{remark}
\begin{proof} 
We will prove the lemma by induction on $k$. Let $\{e_j\}$ be the canonical base of $\Omega$. The entries of $e_j$ satisfies: $(e_j)_k=\delta_{\{k=j\}}$. 

For $k=1$, there are $25$ possible positions of the origin within $\region_1$ . Assign to each of those vertices a pair $(a,b)$ given by the distance from it to the bottom side and left side of $\region_1$, respectively. This is a surjective map from the set of vertices in $\region_1$ and $\{0,1,2,3,4\}^2$. We can check now that the origin the vertex with label $(a,b)$ if and only if: $(\omega_1)_1=a$ and $(\omega_2)_1=b$. This proves the initial case. To prove the general case, consider $\region_k$ divided into $25$ squares of side $5^{k-1}$. We will prove next that the pair $((\omega_1)_k,(\omega_2)_k)$ is enough to determine in which of these squares the origin is. To see this, we argue similarly to the case $k=1$. Notice that each of the $25$ squares can be encode by a pair $(a,b)$ given by the distance to the bottom and left side of $\region_k$, respectively. We can check that the origin lies in the square labeled $(a,b)$ if and only if $(\omega_1)_k=a$ and $(\omega_2)_k=b$. Using the induction hypothesis the proof will follow. 
\end{proof}

\begin{lemma}\label{corner_distribution}
Denote by $\theta^k_1$ the argument of $v^k_1$.
 Fix $\theta\in (0,\pi/2)$. For any $\epsilon>0$ there are infinitely many values of $k$ such that

$$|\theta-\theta^k_1|< \epsilon.$$
\end{lemma}

\begin{proof}
We will do the case $\theta=0$. We want to show that infinitely many $v^k_1$ are inside the cone bounded by the lines $\theta=0$ and $\theta=\epsilon$. Let $t$ be a natural number such that $\frac{5^{-t}}{1-5^{-t}}<\tan(\epsilon)$. For large values of $k$, denote by $E_k$ the event:
\[
(\omega_1)_{k-t+1}=(\omega_1)_{k-t+2}=\cdots=(\omega_1)_{k}=4
\]
\[
(\omega_2)_{k-t+1}=(\omega_2)_{k-t+2}=\cdots=(\omega_2)_{k}=0.
\]
In words, this corresponds to $t$ coordinates been simultaneously equal to $0$ and $4$ in $\omega_1$ and $\omega_2$, respectively. If follows by Borel-Cantelli that $\{E_k\}$ happens infinitely often. By Lemma \ref{origin_position}, this event corresponds to the origin being in the top left $5^{k-t}$ square in $\region_k$. Then
\[
0<\theta^k_1\leq \arctan\left(\frac{5^{k-t}}{5^k-5^{k-t}}\right)<\epsilon.
\]  
which completes the proof.
\end{proof}

\subsection{Proof of Theorem \ref{dir_rays}.}

Let $\rho$ be a functional tangent to the $\ell_1$--ball. Associate to $\rho$ a set $C_i$ of corners, in the natural way. By Lemma \ref{dh lemma} there is a unique geodesic ray, $\gamma_i$, with the property that $Dir(\gamma_i)\subset (i-1)(\pi/2)+(0,\pi/2)$. By Lemma \ref{corner_distribution} we can find points  $u\in Dir(\gamma_i)$ as close as we want to the endpoints of $(i-1)(\pi/2)+(0,\pi/2)$. Since $Dir(\gamma_i)$ is connected we conclude that $(i-1)(\pi/2)+(0,\pi/2)\subset Dir(\gamma_i)$. It follows now that $(i-1)\pi/2+[0,\pi/2]=Dir(\gamma_i)$.

\section{Exponents in non-coordinate directions}\label{sec:var}

The next two sections are devoted to the proof of Theorem \ref{Thm3}. We start by showing that $T(x,y)$ is well concentrated. Denote the origin by $\0$.

\begin{lemma} \label{naomi}
Let $0<\lambda\leq 1$ be a fixed constant and $\slope=(n,\lambda n)$. There exists a  constant $C=C(\lambda)$ such that
\begin{equation*}
|\slope|_1\leq T(\0,\slope)\leq |\slope|_1+C.
\end{equation*}
\end{lemma}
\begin{proof}
The lower bound follows from the fact that $t_e\geq 1$. For the upper bound, we construct a path from $\0$ to $\slope$ satisfying the desired inequality. Consider the squares $\{\region_k(\0)\}$ and  
$\{\region_k(\slope)\}$ for $1\leq k\leq N-1$ where 
$N$ is the minimum $t$ such that the projections of 
$\region_t(\0)$ and $\region_t(\slope)$
onto either the $x$ or $y$ axes have nonempty intersection.
Note that this definition implies that $n \leq (2/\lambda) 5^N.$ 
 
Next, we choose a corner in $\{\region_k(\0)\}$ for each $k$ that is closest to $\slope$, and similarly we choose a corner in $\{\region_k(\slope)\}$ closest to $\0$. We have a sequence of vertices:
\begin{equation*}
\0, v_1, v_2,\dots, v_{N-1}, w_{N-1},\dots, w_2, w_1, \slope
\end{equation*}
where $v_i, w_i$ are the corners chosen above in $\region_i(\0)$ and $\region_i(\slope)$, respectively. Our path is the concatenation of paths joining consecutive vertices on this sequence such that the path connecting $v_i$ and $v_{i+1}$ (and $w_{i+1}$ and $w_i$) is on the $i-$grid, and the path joining $v_{N-1}$ and $w_{N-1}$ is on the $(N-1)-$grid. All of these subpaths are taken to have the minimal possible number of edges.

Note that for each $i$, all edges in the subpaths from $v_i$ to $v_{i+1}$ have passage times bounded by $1+2\alpha^{i}$. Also, we cross at most $2\cdot5^{i+1}$ many edges between $v_i$ and $v_{i+1}$. An analogous analysis extends to the vertices $w_j$. 
Between $v_{N-1}$ and $w_{N-1}$ we have at most $2n$ edges with weights at most 
$1+2\alpha^{N-1}$. The total length of our path is at most $|\slope|_1$.
We put this together to conclude that
\begin{equation*}
T(\0,\slope)
\leq |\slope|_1+20\sum_{j=1}^{\infty} 5^j\alpha^j +n 2\alpha^N
\leq |\slope|_1+C' +(2/\lambda)5^N 2\alpha^N
=|\slope|_1+ C
\end{equation*}
as $\alpha<1/5$.
\end{proof}

\subsection{The wandering exponent.}\label{sec:wand}

\begin{proposition}\label{prop:exp2} \label{osaka}

Let $0<\lambda \leq 1$. 
Let $\cyl(\slope,cn)$ 
be the set of all points within
distance $cn$ of the line segment connecting $\0$ and $\slope=(n,\lambda n)$.
There exists $c=c(\lambda)>0$ such that for all $n$ sufficiently large $\geo(\0,(n,\lambda n))$ is not contained in 
$\cyl(\slope,cn)$.

\end{proposition}

\begin{proof}
Let $P$ be a path from $\0$ to $(n, \lambda n)$ with all of its vertices in $\cyl(\slope,cn)$. 
Let $j=\lfloor \log_5(\lambda n)\rfloor -2$. Let $j'$ be the smallest integer such that no horizontal (or vertical) line segment of length $5^{j'-3}$ lies entirely in $\cyl(\slope,cn)$. Note that if $c$ is small and $n$ is sufficiently large then $j>j'+5$.
Let 
$z_1$ be the first (closest to $\0$) vertex of $P$ in the $j-$grid
and let 
$z_2$ be the last (closest to $(n,\lambda n)$) vertex of $P$ in the $j-$grid.
Note that by the choice of $j$ we have that both the $x$ and $y$ coordinates of $z_2$ are at least $5^j$ 
greater than the respective $x$ and $y$ coordinates of $z_1$. Thus there exists a northeast directed path $P''$ 
from $z_1$ to $z_2$ that is contained entirely in the $j-$grid.
We will show that there exists a path $P'$ which is not contained in $\cyl(\slope,cn)$
which is a faster path from $\0$ to $(n, \lambda n)$.
$P'$ will agree with $P$ from $\0$ to $z_1$ and from $z_2$ to $(n, \lambda n)$.
Between $z_1$ and $z_2$ the path $P'$ is $P''$. The choice
of $j$ insures that this is possible and $|z_1-z_2| >\frac {1+\lambda}{2}n$.

As every edge of $P'$ between $z_1$ and $z_2$ is in the $j-$grid, the sum of the
passage times of all of these edges is at most
$$|z_1-z_2|(1+2 \cdot \alpha^j).$$
We now show that this is faster than $P$  so this path  is not a geodesic.

Define a sequence $\{z_1^i\}_{i=0}^k$ with $z_1^0=z_1$ and $z_1^k=z_2$ with
each $z_1^i$ (with $0<i<k$) the first time that $P$ hits a new vertical line on the $j'-$grid. 
Note that $k$ is at least $5^5>1000.$
If $0<i<k-1$ and the path $P$
between $z_1^i$ and $z_1^{i+1}$ hits another vertical line (besides the start and end lines) in the
$j'-$grid then it has at least $3\cdot 5^{j-1}$ horizontal edges. Similarly we can see that between $z_1^i$ and $z_1^{i+1}$ the path $P$ hits at two horizontal lines in the ${(j-1)}-$grid.
 By the choice of $j'$ and $\lambda \leq 1$ we have 
$|z_1^i-z_1^{i+1}| \leq 3 \cdot 5^{j'}$ and the difference in the $x$ coordinate is $5^{j'}$. As all edges have passage times between 1 and 2 then $P$ is not a geodesic. 

Otherwise as $\lambda \leq 1$ this segment from $z_1$ to $z_2$ of $P$ contains edges on the $j'-$grid on at most one vertical and two horizontal lines. By the choice of $j'$ each of these lines contains at most $5^{j'-3}$ edges in $\cyl(\slope,cn)$. Thus this segment of $P$ contains at most $3\cdot 5^{j'-3}$ edges in the $j'-$grid and at least $5^{j}$ edges in total. Thus at least 85\% of the edges in this segment of $P$ are not in the $j'-$grid and have passage times at least $1+\alpha^{j'}$.
As this applies to all but the first and last segments, at least 80\% of the edges on $P$ from $z_1$ to $z_2$ are not in the $j'-$grid. As above the first and last segments have at most $3\cdot 5^{j-3}$ edges
and thus 
 make up less than three percent of the length of $P$ from $z_1$ to $z_2$ (see Figure 2).

 \begin{figure}[h]\label{fig2:prop2}
\includegraphics[scale=.4]{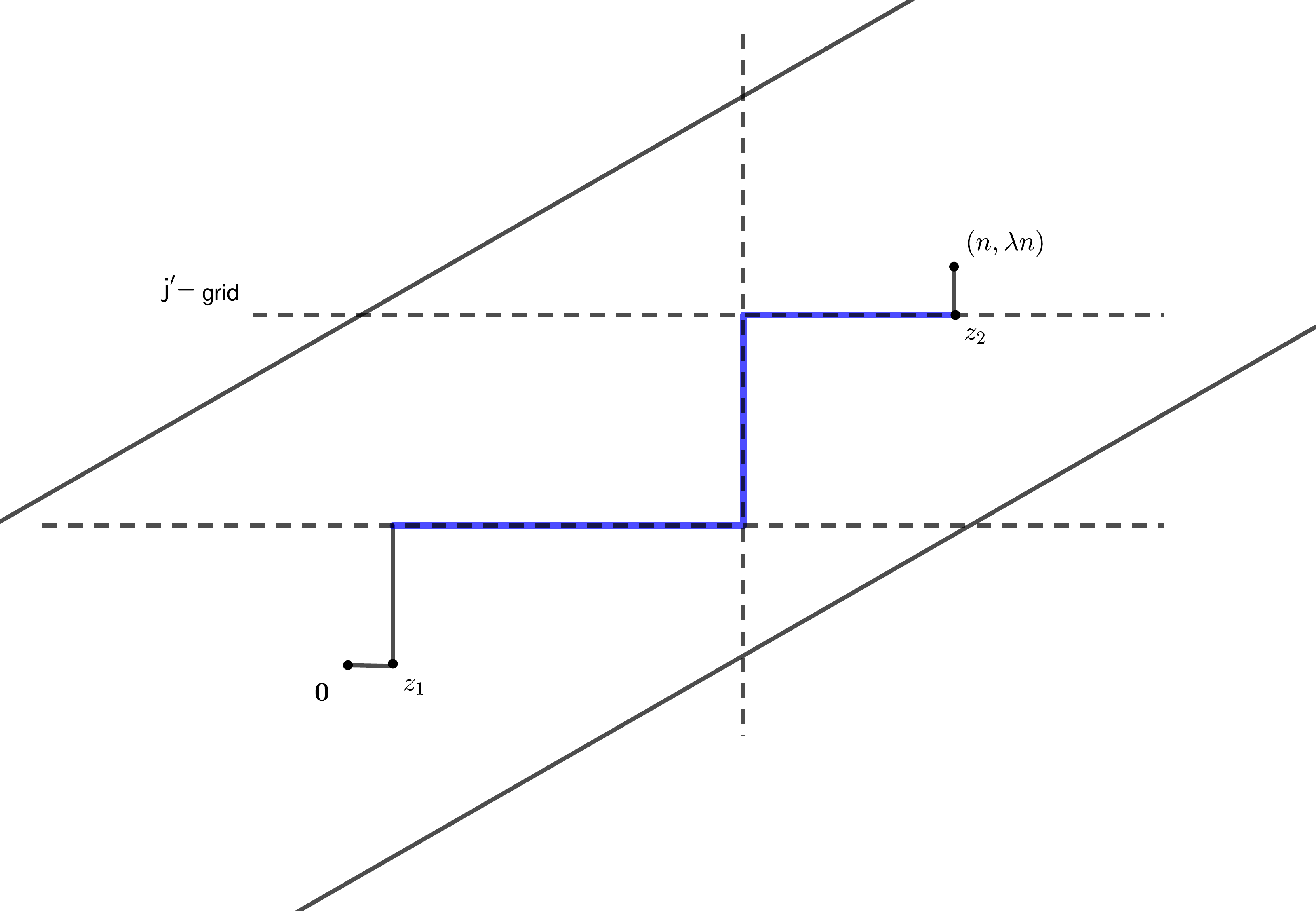}

\caption{Proof of proposition \ref{prop:exp2}: the path $P’$ hits one vertical and two horizontal lines in the $j’-$grid. Now, by definition of $j’$, at most $3\cdot 5^{j’-3}$ edges in the geodesic from $z_1$ to $z_2$ are on the $j’-$grid. Furthermore, by the choice of $j, z_1$ and $z_2$ most of the edges of the path $P$ are in the segment connecting these two vertices. We then put all these ingredients together to lower bound the passage time of the entire path.}       
\end{figure}

Thus the total passage time for $P$ between $z_1$ and $z_2$ is at least
$$|z_1-v_2|(1+.8\cdot\alpha^{j'}) > |v_1-v_2|(1+2 \alpha^{j}).$$
Thus the passage time along $P$ is more than the passage time along $P'$ and $P$ is not a geodesic.
This proves that the geodesic does not lie in $\cyl(\slope,c n)$.

\end{proof}

\begin{proposition}\label{prop:exp3}
Let $0<\lambda \leq 1$. Remember that $\cyl(\slope,10n)$ is the set of all points within
distance $10n$ of the line segment connecting $\0$ and $(n,\lambda n)$.
For all $n$ sufficiently large $\geo(\0,(n,\lambda n))$ is contained in $\cyl(\slope,10n)$.
\end{proposition}

\begin{proof}
If a path $P$ from $\0$ to $(n, \lambda n)$ is not in $\cyl(\slope,10n)$ then the length of $P$ is at least $10 n$.
But as $\lambda \leq 1$ there is a path $P'$ which is in $\cyl(\slope,cn)$ from $\0$ to $(n, \lambda n)$ of length at most $2n$. As every
edge has weight at most $2$ the length of $P'$ is at most $4n$ and $P$ is not the geodesic from
$\0$ to $(n, \lambda n)$.
\end{proof}

\section{Exponents in the coordinate direction} \label{sec:cor}

In this section we consider  $\gamma(\0,(n,0))$. 
Define
\begin{equation} \label{groupon}
\beta=\frac{\log 5}{\log 5-\log \alpha}<1.
\end{equation}
Note that, for any $j$ we can write
\begin{equation} \label{veins}
\alpha^{\beta j}=(5^j)^{\beta-1}.
\end{equation}

\begin{lemma} \label{sunset}
There exists universal constants $C$ and $N$ such that for all $c>C$ and $n>N$ we have

\begin{enumerate}
\item $\tn \geq n$ \label{sunrise}
\item $\P(\tn \leq n+.01n^\beta)>10^{-9}$  \label{fast}
\item $\P(\tn \geq n+.02n^\beta)>10^{-9}$ and \label{ramadan}
\item $\P(\tn \geq n+10n^\beta)=0$. \label{snoop dogg}
\end{enumerate}
\end{lemma}

\begin{proof}
The first inequality is true because all passage times are at least 1.

For the second inequality we
define $\Gamma_l$ to be the following path from $(0,0)$ to $(n,0)$. 
The start of $\Gamma_l$ goes northeast
from $(0,0)$ to the line $y=y_l$, where $y_{l}$ is the lowest non-negative number such that the
line $y=y_{l}$ is in the $l$-grid.
Suppose we have defined the path to the point $(x',y')$ where both the lines $x=x'$ and $y=y'$ are in the $l'$-grid. Then we extend the path so that it goes east to the $(l' +1)$-grid and then north to the $(l'+1)$-grid. We continue until we have hit the line $y=y_l$. 
The final portion of $\Gamma_l$ is defined in a symmetric manner. It goes northwest from $(n,0)$ to the line $y=y_{\ell}$. Then $\Gamma_l$ connects these two pieces by moving horizontally along the line $y=y_{\ell}$. 

Given $n$, choose $j$ such that 
\begin{equation} \label{blowout} 5^j \leq n <5^{j+1}. \end{equation}
Let $\strike$ be the event that there exists $y^* \in [0, .001\cdot 5^{\beta j}]$ with the line $y=y^*$ in the 
$(\lfloor  \beta j \rfloor +5)$-grid. If $\strike$ occurs then $\Gamma_{\lfloor \beta j  \rfloor +5}$ contains:
\begin{enumerate}
\item at most $n+ .002 \cdot 5^{\beta j}$ edges,
\item at most $20 \cdot 5^k$ edges in the $k$-grid but not the $(k+1)$-grid for all $k <\lfloor \beta j \rfloor +5$, and, 
\item at most $n$ edges in the $(\lfloor \beta j \rfloor +5)$-grid.
\end{enumerate}

If $\strike$ occurs, from $1-3$ above and the definition of the $X_{k(e),e}$, we have 

\begin{eqnarray*}
\sum_{e \in \Gamma_{\lfloor \beta j \rfloor +5 }}\alpha^{k(e)}+X_{k(e),e}
& \leq & \sum_{e \in \Gamma_{\lfloor \beta j \rfloor +5 }}1.001\alpha^{k(e)}\\
& \leq & 1.001 \sum_k \left( 20 \cdot 5^k \alpha^k\right) + 1.001n \alpha^{\lfloor \beta j\rfloor +5}\\
& \leq & C+ .001 \cdot 5^{\beta j-1}\\
& \leq & C+ .001n^{\beta}.
\end{eqnarray*}
Then, if $\strike$  occurs
\begin{eqnarray*}
T((0,0),(n,0)) &\leq& T(\Gamma_{\lfloor \beta j  \rfloor+5})\\
 & \leq &| \Gamma_{\lfloor \beta j  \rfloor+5}| + \sum_{e \in \Gamma_{\lfloor \beta j \rfloor +5 }}\alpha^{k(e)}+X_{k(e),e}\\
 & \leq & n  + .002 \cdot 5^{\beta j} +C+.001n^{\beta} \\
 & \leq & C+n + .004n^{\beta}\\
 & \leq & n + .01n^{\beta}.
\end{eqnarray*}
Then
$$\P(\tn \leq n+.01n^\beta) \geq \P(\strike)\geq .001\cdot 5^{-5} \geq 2 \cdot 10^{-9}$$
and the result follows.

The fourth inequality follows in much the same way as the second
except we do not assume that the
event $\strike$ occurs. In this case we have that $\Gamma_{\lfloor \beta j  \rfloor }$ contains
\begin{enumerate}
\item at most $n+ 2 \cdot 5^{\beta j}$ edges
\item at most $20 \cdot 5^k$ edges in the $k$-grid but not the $(k+1)$-grid for all $k <\lfloor \beta j \rfloor $ and 
\item at most $n$ edges in the $(\lfloor \beta j \rfloor)$-grid.
\end{enumerate}
Then a similar calculation as above proves the claim.

For the third inequality we note that if there does not exist $y_0$ such that
$|y_0| \leq .1\cdot 5^{\beta j}$ such that the line $y=y_0$ is in the $\lfloor \beta j \rfloor$-grid and if
$\Gamma'$  be any path from $(0,0)$ to $(n,0)$ in the cylinder $\{(x,y): |y| \leq  .1\cdot 5^{\beta j}\}$,
then
\begin{eqnarray*}
T(\Gamma')
 & \geq & n(1+\alpha^{\beta j -1})\\
 & \geq & n +\frac1\alpha n \alpha^{\beta j }\\
 & \geq & n +\frac1\alpha n (5^j)^{\beta -1 }\\
 & \geq & n +\frac1\alpha n (5^j)^{\beta } (5^j)^{-1 }\\
 & \geq & n +\frac1\alpha  (5^j)^{\beta } \\
 & \geq & n +\frac1\alpha  (5^{j+1})^{\beta }5^{-\beta} \\
 & \geq & n +\frac1{\alpha 5^\beta} n^{\beta } \\
 & \geq & n + n^{\beta }.
\end{eqnarray*}
Now let $\Gamma''$  
be any path from $(0,0)$ to $(n,0)$ not contained in the cylinder 
$\{(x,y): |y| \leq  .1\cdot 5^{\beta j}\}.$
Then by \eqref{groupon} and \ref{blowout}
\begin{eqnarray*}
T(\Gamma'')
 & \geq & n+.2\cdot 5^{\beta j}\\
 & \geq & n+.2\cdot (5^{j+1})^{\beta}5^{-\beta}\\
 & \geq & n+.04 n^{\beta }.
\end{eqnarray*}
As any path from $(0,0)$ to $(n,0)$ falls into one of these two categories we have  that
 $$T((0,0),(n,0)) \geq  \min( n+ n^{\beta},n+.04 n^{\beta } )=  n+.04 n^{\beta }.$$
This happens with probability at least 
$$1- \frac{.3\cdot 5^{\beta j}}{5^{\lfloor \beta j \rfloor}} \geq 10^{-9}.$$

\end{proof}

We use Lemma \ref{sunset} to show that the variance exponent is $\beta$ along the axes.
\begin{lemma} \label{thief}
There exists $K>0$ such that for all $n$ sufficiently large
$$ \frac1K n^{2\beta} < Var(\tn)<K n^{2\beta} .$$
\end{lemma}

\begin{proof}
The lower bound follows directly from parts  \ref{fast} and \ref{ramadan} from Lemma \ref{sunset}.
The upper bound follows from parts \ref{sunrise} and \ref{snoop dogg}.
\end{proof}

For any $K$ define $\cyl((n,0),K)$ be the subgraph with vertices  $\{(x,y):\ |y| \leq K \}$ and all edges between two vertices in the set.
Now we show that the fluctuation exponent is also $\beta$.
\begin{lemma} \label{coaching} For any $\epsilon>0$
$$ \P(\gamma(\0,(n,0)) \text{ is contained in } \cyl((n,0),n^{\beta-\epsilon})=o(1).$$
Also
$$ \P(\gamma(\0,(n,0)) \text{ is not contained in } \cyl((n,0),10 n^{\beta})=0.$$
\end{lemma}

\begin{proof}

Define $$j=j(n)=\max\{k(e):~\mbox{$e$ is an horizontal edge in $\gamma(\0,(n,0))$}\}.$$

We first notice that all horizontal edges in $\gamma(\0,(n,0))$ with $k(e)=j$ are contained in the horizontal line that is furthest away from the $x-$axis. Consider a path $P$ that goes up to the ${j+1}$ grid and connects $\0$ and $(n,0)$. 
We have
\[
T(\0,(n,0))\leq T(P).
\]
by definition.

For any $\epsilon>0$ and for all $n$ sufficiently large we will show that
$$ \P(\exists \text{ a path $P$ from $\0$ to $(n,0)$ contained in } \cyl((n,0),n^{\beta-\epsilon})\text{ with } T(P) \leq n+ 10n^{\beta })=o(1).$$

There are at least $n$ horizontal edges in any path from $\0$ to $(n,0)$. If all the horizontal edges in 
$\cyl((n,0),n^{\beta-\epsilon})$ have passage time at least $1+10n^{\beta-1}$ then the passage time across any path 
from $\0$ to $(n,0)$ entirely contained in
$\cyl((n,0),n^{\beta-\epsilon})$ has passage time at least $n+10n^\beta$.

By part \ref{snoop dogg} of Lemma \ref{sunset} the event that 
$$
\text{$\gamma(\0,(n,0))$ is contained in } \cyl((n,0),n^{\beta-\epsilon})
$$
is contained in the event that
$$
\text{there exists a path $P$ from $\0$ to $(n,0)$ contained in } \cyl((n,0),n^{\beta-\epsilon})\text{ with } T(P) \leq n+ 10n^{\beta }.
$$
This last event is in turn contained in the event that 
$$
\text{there exists a horizontal edge in $\cyl((n,0),n^{\beta-\epsilon})$ with passage time at most $1+ 10n^{\beta-1}$.}
$$ 
This requires that there is a line of the form $y=l$ which is in the $(\lfloor \beta j\rfloor-3)$-grid 
with $l \in [-n^{\beta-\epsilon},n^{\beta-\epsilon}]$. 
By the choice of $\beta$ and $j$ the probability of this is at most 
$$\frac{2n^{\beta-\epsilon}+1}{5^{\beta j-3}} \leq Cn^{-\epsilon}.$$

The upper bound follows from part \ref{snoop dogg} of Lemma \ref{sunset} and the fact that all passage times
are at least 1.
\end{proof}

\subsection{Proof of Theorem \ref{Thm3}}

For the non-coordinate directions, $\chi=0$ follows directly from Lemma \ref{naomi} and $\xi=1$ follows combining Proposition \ref{osaka} and Proposition \ref{prop:exp3}. For the coordinate directions, $\chi=\xi=\beta$ is a consequence of Lemma \ref{thief} and Lemma \ref{coaching}.

\begin{acknowledgement}
The authors are grateful to the anonymous referee whose suggestions improved the presentation of our work.  G.B. thanks Ken Alexander and Michael Damron for fruitful conversations at the early stages of this project. C.H. was supported by grant DMS-1712701.
\end{acknowledgement}

\input{referenc}
\end{document}

%% file: referenc.tex
%
%
\bibliographystyle{amsalpha}
 \bibliography{fpp}